\documentclass[11pt]{amsart}
\usepackage{geometry}                
\usepackage{graphicx}
\usepackage{amssymb}
\usepackage{epstopdf}
\newtheorem{theorem}{\bf Theorem}[section]
\newcommand{\lra}{\longrightarrow}
\newcommand{\CC}{\mathbb C}

\newcommand{\cO}{\mathcal{O}}
\newcommand{\cL}{\mathcal{L}}

\newcommand{\BNone}{B(1,d,k)}
\newcommand{\Gone}{G(1,d,k)}
\newcommand{\BN}{B(r,d,k)}
\newcommand{\Galpha}{G(\alpha;r,d,k)}
\newcommand{\GL}{G(\alpha;r,\cL,k)}
\newcommand{\GLi}{G_i(r,\cL,k)}
\newcommand{\GLL}{G_L(r,\cL,k)}
\newcommand{\GLO}{G_0(r,\cL,k)}
\newcommand{\GLr}{G_L(r,\cL,r)}
\newcommand{\GLrr}{G_L(r,\cL,r+1)}
\newcommand{\bnone}{\beta(1,d,k)}
\newcommand{\bn}{\beta(r,d,k)}
\newcommand{\BNl}{B(r,\cL,k)}
\newcommand{\Gr}{\operatorname{Gr}}

\newtheorem{prop}[theorem]{Proposition}
\newtheorem{cor}[theorem]{Corollary}
\newtheorem{conj}[theorem]{Conjecture}
\newtheorem{rem}[theorem]{Remark}
\newtheorem{ex}[theorem]{Example}

\numberwithin{equation}{section}

\title{On coherent systems with fixed determinant}
\author{I. Grzegorczyk and P. E. Newstead}
\address{Department of Mathematics, CSU Channel Islands, One University Drive, Camarillo CA 93012, USA}
\email{ivona.grzegorczyk@csuci.edu}

\address{Department of Mathematical Sciences, The University of Liverpool,
Peach Street, Liverpool L69 7ZL, UK}
\email{newstead@liv.ac.uk}
\thanks{Both authors are members of the research group Vector Bundles on Algebraic Curves (VBAC). They would like to thank the Isaac Newton Institute, where the first draft of this paper was written during the Moduli Spaces programme in 2011.}
\date{\today}

\begin{document}
\begin{abstract}
Over the past 20 years, a great deal of work has been done on the moduli spaces of coherent systems on algebraic curves. Until recently, however, there has been very little work on the fixed determinant case, except for the special case of rank 2 and canonical determinant. This situation has changed due to two papers of B.~Osserman, who has obtained lower bounds for the dimensions of the fixed determinant moduli spaces in some cases. Our object in this paper is to show that some of Osserman's bounds are sharp.
\end{abstract}
\maketitle
\section{Introduction}\label{intro}

Let $C$ be a smooth projective curve of genus $g$ over $\CC$ and let $J^d(C)$ denote the Jacobian of line bundles on $C$ of degree $d$. In classical Brill-Noether theory, one considers {\it linear systems} $(\cL,V)$, where $\cL$ is a line bundle of degree $d$ and $V$ is a linear subspace of $H^0(C,\cL)$ of dimension $k$. The {\it Brill-Noether locus} $\BNone$ (denoted classically by $W^{k-1}_d$) is defined by
$$\BNone:=\{\cL\in J^d(C)\,|\,h^0(\cL)\ge k\},$$
where $h^0(\cL):=\dim H^0(\cL)$.
We write also 
$$\Gone:=\{(\cL,V)\,|\,\cL\in J^d(C),\dim V=k\}$$
for the corresponding variety of linear systems $(\cL,V)$ (this is denoted classically by $G^{k-1}_d$). Clearly there is a morphism $\Gone\to \BNone$, whose fibre over $\cL$ is the Grassmannian $\mbox{Gr}(k,h^0(\cL))$. We write also 
$$\bnone:=g-k(k-d+g-1),$$
a number which is known as the {\it Brill-Noether number} or the {\it expected dimension} of $\Gone$.

The main results of the classical theory (see \cite{acgh}) are most easily stated in terms of $\Gone$ and are as follows:
\begin{itemize}
\item[(i)] every irreducible component of $\Gone$ has dimension at least $\bnone$;
\item[(ii)] $\Gone$ is smooth of dimension $\bnone$ at $(\cL,V)$ if and only if the {\it Petri map} $V\otimes H^0(\cL^*\otimes K_C)\to H^0(K_C)$, given by multiplication of sections, is injective;
\item[(iii)] $\Gone\ne\emptyset$ if $\bnone\ge0$;
\item[(iv)] for general $C$, $\Gone$ is empty if $\bnone<0$, smooth of dimension $\bnone$ if $\bnone\ge0$ and irreducible if $\bnone>0$.
\end{itemize}
These statements together justify the use of the term ``expected dimension''. The proof of (i) depends on representing $\BNone$ as a degeneracy locus while (ii) is proved by identifying the Zariski tangent space of $\Gone$ in terms of the Petri map; (iv) is proved by showing that, for general $C$, the Petri map is injective for every $(\cL,V)$.

Replacing $\cL$ by a vector bundle $E$ of rank $r$, we can define
$$\BN:=\{E\in M(r,d)|h^0(E)\ge k\},$$
where $M(r,d)$ is the moduli space of stable bundles of rank $r$ and degree $d$. The pairs $(E,V)$ are now called {\it coherent systems}; we can consider the moduli of these either as a stack or as a moduli space with a suitable stability condition (see section \ref{back} for further details). It turns out that the analogues of (i) and (ii) are true but not those of (iii) and (iv); in particular, the non-emptiness of the moduli space (even for a general curve) is a delicate question which is far from being resolved. 
However, the analogues of (iii) and (iv) are true in sufficiently many cases that the {\it Brill-Noether number}
$$\bn:=r^2(g-1)+1-k(k-d+r(g-1))$$
can reasonably be regarded as the expected dimension of the moduli space.

Suppose now that we fix the determinant of $E$. The Brill-Noether locus can still be represented as a degeneracy locus and every component has dimension at least $\bn-g$; however, this lower bound on dimension is often not best possible and cannot in general be sensibly regarded as the expected dimension. One case in which almost everything works well is when $E$ has rank $2$ and $\det E=K_C$; there is then a modified Brill-Noether number and the only unsettled issue is (iii). There are quite strong results in this direction (see \cite{te3,lnp}), but there are values of $g$ and $k$ for which the Brill-Noether number is non-negative but it is unknown whether the Brill-Noether loci and the moduli spaces of coherent systems are non-empty.

For $\det E\not\simeq K$, the situation has changed recently due to work of Osserman \cite{o1,o2}. For rank $2$ and some cases in higher rank, he has obtained improved lower bounds on the dimension of the moduli stack, which lead to a general conjecture. He gives also some examples in which these bounds are sharp in the sense that there exists a component of the moduli stack of the modified expected dimension. It is the purpose of the present note to give more general examples of this phenomenon.

In section \ref{back}, we define $\alpha$-stability for coherent systems and review the results on lower bounds for the dimensions of moduli spaces of $\alpha$-stable coherent systems with fixed determinant. We also formulate a conjecture (Conjecture \ref{conj}) which is compatible with these results. Section \ref{summ} contains our main results, which can be summarised by saying that the lower bounds of Conjecture \ref{conj} are sharp when $k\le r+1$. We conclude with some examples in section \ref{ex}.

We work throughout over a smooth projective curve $C$ defined over the complex numbers and denote by $K_C$ the canonical line bundle on $C$.

Our thanks are due to Montserrat Teixidor i Bigas for some helpful comments and to the referee for a careful reading of the paper and some useful observations.

\section{Lower bounds}\label{back}
 
We recall that a {\it coherent system} on $C$ of type $(r,d,k)$ is a pair $(E,V)$ consisting of a vector bundle $E$ of rank $r$ and degree $d$ and a linear subspace $V$ of $H^0(E)$ of dimension $k$. We shall suppose always that $k\ge1$. For any $\alpha\in{\mathbb R}$, the $\alpha$-{\it slope} of $(E,V)$ is defined by
$$\mu_\alpha(E,V):=\frac{d}r+\alpha\frac{k}r.$$
The coherent system $(E,V)$ is $\alpha$-{\it stable} ($\alpha$-{\it semistable}) if, for every proper coherent subsystem $(F,W)$ of $(E,V)$,
$$\mu_\alpha(F,W)<(\le)\mu_\alpha(E,V).$$
Some necessary conditions for the existence of $\alpha$-stable coherent systems are
\begin{equation}\label{nec}
d>0,\ \alpha>0,\ (r-k)\alpha<d.
\end{equation}

There exists a moduli space for $\alpha$-stable coherent systems of type $(r,d,k)$, which we denote by $\Galpha$. As already noted in the introduction, we define the {\it Brill-Noether number} (or {\it expected dimension}) of $\Galpha$ by
$$\bn:=r^2(g-1)+1-k(k-d+r(g-1)).$$
We then have
\begin{itemize}
\item[(i)] every irreducible component of $\Galpha$ has dimension at least $\bn$;
\item[(ii)] $\Galpha$ is smooth of dimension $\bn$ at $(E,V)$ if and only if the {\it Petri map} $V\otimes H^0(E^*\otimes K_C)\to H^0(E\otimes E^*\otimes K_C)$, given by multiplication of sections, is injective.
\end{itemize}
(For these and many other facts about coherent systems and their moduli, see \cite{BGMN, BGMMN1, BGMMN2}.)

\medskip
Now suppose we fix the determinant $\cL$ of $E$ and define moduli spaces $\GL$. These are closed subschemes of $\Galpha$, where $\deg\cL=d$.

\begin{prop}\label{prop1}
Every irreducible component $X$ of $\GL$ has dimension  
\begin{equation}\label{naive}
\dim X\ge\bn-g.
\end{equation}
\end{prop}
\begin{proof} This follows at once from (i) above.
\end{proof}

It is easy to see that this estimate is not always best possible. Consider, for instance, the trivial case $r=1$. Then $\GL$ is just the Grassmannian $\Gr(k,h^0(\cL))$, which has dimension
\begin{equation}\label{r=1}
k(h^0(\cL)-k)=k(-k+d-g+1+h^1(\cL))=\bnone-g+kh^1(\cL),
\end{equation}
whenever it is non-empty.

A more significant example is given by the case $r=2$, $\cL=K_C$.

\begin{ex}\label{ex1}\begin{em}
Let $X$ be an irreducible component of the moduli space $G(\alpha;2,K_C,k)$. Then (compare \cite[Section 2]{bf} or \cite[Theorem 4.2]{m} or see \cite[Theorem 1.1]{o1})
\begin{equation}\label{kc}
\dim X\ge3g-3-\frac{k(k+1)}2=\beta(2,2g-2,k)-g+\left(\begin{array}{c}k \\2\end{array}\right).
\end{equation}
 Now let $(E,V)\in G(\alpha;2,K_C,k)$ and let $S$ denote the image of $V\otimes H^0(E)$ in $S^2H^0(E)$. The infinitesimal behaviour of $G(\alpha;2,K_C,k)$ at $(E,V)$ is governed by the modified Petri map $\mu:S\to H^0(S^2E)$. In particular, combining \cite[Th\'eor\`eme 3.12]{he} with the results of \cite{bf,m}, one can show, provided $h^1(S^2E)=0$ and $\mu$ is injective, that the dimension of the Zariski tangent space to $G(\alpha;2,K_C,k)$ is 
\begin{eqnarray*}
3g-3-\dim S+\dim\mbox{Gr}(k,h^0(E))&=&3g-3-\left(kh^0(E)-\frac{k(k-1)}2\right)+k(h^0(E)-k)\\&=&3g-3-\frac{k(k+1)}2.
\end{eqnarray*}
It follows then that $G(\alpha;2,K_C,k)$ is smooth of precisely the dimension given by \eqref{kc} at $(E,V)$. For $C$ general and $E$ semistable, it is proved in \cite{te1} that $\mu$ is injective. Moreover, if $E$ is stable, then $h^1(S^2E)=0$ since $S^2E\simeq End\,^0(E)\otimes K$. So, for general $C$, $G(\alpha;2,K_C,k)$ is smooth of precisely the dimension given by \eqref{kc} at $(E,V)$ whenever $E$ is stable.
\end{em}\end{ex}

Based on these examples, it is reasonable to make the following conjecture.

\begin{conj}\label{conj} Every irreducible component $X$ of $\GL$ has dimension
\begin{equation}\label{conjeq}
\dim X\ge\bn-g+\left(\begin{array}{c}k \\r\end{array}\right)h^1(\cL).
\end{equation}\end{conj}
We have already seen that this conjecture holds for $k<r$ by Proposition \ref{prop1}, $r=1$ by \eqref{r=1} and $r=2$, $\cL=K_C$ by Example \ref{ex1}. It has recently been proved by Osserman  \cite{o1} that it holds also when $r=2$ and $h^1(\cL)\le2$. In \cite[Theorem 1.1]{o2}, Osserman proves further that \eqref{conjeq} holds for any component of $\GL$ passing through $(E,V)$ in the following cases.
\begin{itemize}
\item $k=r$, $V$ not contained in a subbundle of $E$ of rank $r-2$;
\item $k=r+1$, $h^1(\cL)=1$, no $r$-dimensional subspace of $V$ contained in a subbundle of $E$ of rank $r-2$;
\item $r=3$, $k=5$ or $6$, $h^1(\cL)=1$, no subspace of $V$ of dimension $2$ contained in a line subbundle of $E$.
\end{itemize}
Note that the first case here proves the conjecture completely when $k=r=2$. In general, however, Osserman notes that some non-degeneracy condition is required. He works in terms of stacks and proves that, for any $r$, $k$ and $\cL$, there is an open substack of the moduli stack for which the bound of the conjecture holds; the problem is to describe this open substack, which a priori may even be empty \cite[section 5]{o2}. In view of Osserman's examples, it remains possible that $\alpha$-stability is a sufficient non-degeneracy condition, so that the conjecture in the form stated above remains open. Some evidence for this is provided in the next section. 

Regarding the sharpness of the bound in Conjecture \ref{conj}, Teixidor \cite{te2} has shown that, for a general curve and $r=2$, the bound \eqref{naive} is sharp for $\BNl$ if $d\le k+2g-2$ and $d$ is sufficiently large compared with $k$ ($d\ge k+2g-2-\frac{2g}k$ if $k$ is even, $d\ge k+2g-2-\frac{2(g-1)}{k+1}$ if $k$ is odd). If we work in terms of coherent systems, the condition $d\le k+2g-2$ is not required.
Moreover, Osserman has proved sharpness of \eqref{conjeq} for $k=r=2$ \cite[Theorem 1.3]{o2}. 

\section{Sharpness of the bounds}\label{summ}

Let us fix integers $r\ge2$, $k\ge1$ and a line bundle $\cL$ of degree $d>0$. Within the allowable range \eqref{nec} for the existence of $\alpha$-stable coherent systems, there are finitely many {\it critical values}
$$\alpha_1<\alpha_2<\cdots<\alpha_L$$
such that the stability conditions change only as $\alpha$ passes through a critical value \cite{BGMN}. We write 
\begin{itemize}
\item $\GLi:=\GL$ for $\alpha_i<\alpha<\alpha_{i+1}$;
\item $\GLL:=\GL$ for allowable $\alpha>\alpha_L$;
\item $\GLO:=\GL$ for $0<\alpha<\alpha_1$.
\end{itemize}
We shall be primarily concerned here with $\GLL$, but will make deductions for other $\GLi$ and for the Brlll-Noether locus $\BNl$ when possible.

For completeness, we begin with the ``trivial'' cases $g=0$ and $g=1$.

\begin{theorem}\label{thm0}Suppose $g=0$ or $1$. Then $\GLi$ is irreducible of dimension $\bn-g$ whenever it is non-empty.
\end{theorem}
\begin{proof}
For $g=0$, this follows at once from \cite[Theorem 3.2]{lng0}; for $g=1$, one needs to check that the proof of \cite[Theorem 4.3]{lng1} remains valid when the determinant is fixed.
\end{proof}
\begin{rem}\begin{em}
For $g=1$, the moduli spaces are non-empty if and only if \eqref{nec} holds and either $\gcd(r,d)=1$ and $\bn\ge1$ or $\gcd(r,d)>1$ and $\bn\ge2$. For $g=0$, the condition $\bn\ge0$ is clearly necessary, but no necessary and sufficient condition is known in general.
\end{em}\end{rem}

\begin{theorem}\label{thm1}Suppose $g\ge2$, $k<r$ and $d\ge \max\{k-(g-1)(r-k),1\}$. Then $\GLL$ is non-empty and irreducible of dimension
$$\bn-g.$$
Moreover, $\GLL$ contains a Zariski open subset which is isomorphic to a fibration over $M(r-k,\cL)$ with fibre $\Gr(k,d+(g-1)(r-k))$. 
\end{theorem}
\begin{proof}
We know from \cite{bg} (see also \cite[Theorem 5.4]{BGMN}) that $G_L(r,d,k)$ contains a Zariski open subset which is isomorphic to a fibration over $M(r-k,d)$ with fibre $\Gr(k,d+(g-1)(r-k))$. Restricting this fibration to $M(r-k,\cL)$, we obtain the last part of the statement and also the fact that $\GLL$ is non-empty if $d\ge \max\{k-(g-1)(r-k),1\}$ and has an irreducible component of dimension $\bn-g$. Moreover (see \cite{bg}), every $(E,V)\in \GLL$ can be written in the form
\begin{equation}\label{tf}
0\lra V\otimes \cO\lra E\lra F\lra0,
\end{equation}
with $F$ semistable. In view of \eqref{naive}, to complete the proof it is sufficient to show that the coherent systems $(E,V)$ which are expressible in the form \eqref{tf} with $F$ strictly semistable and of determinant $\cL$ depend on fewer than $\bn-g$ parameters. This is a simple counting exercise (see for example \cite[Corollary 7.10]{BGMMN1}, noting that, in \cite[Proposition 7.9]{BGMMN1}, one needs to subtract $g$ from the formula to take account of the fixing of the determinant).
\end{proof}

\begin{rem}\begin{em}
If $\gcd(r-k,d)=1$, the Zariski open subset of the theorem is the whole of $\GLL$. Note that $d\ge \max\{k-(g-1)(r-k),1\}$ is precisely the condition for the non-emptiness of $G_L(r,d,k)$.
\end{em}\end{rem}

\begin{cor}\label{cortf}
Under the hypotheses of Theorem \ref{thm1}, suppose further that $\cL$ is a general line bundle of degree $d$. Then, for general $(E,V)\in \GLL$, the bundle $E$ is stable and $(E,V)\in\GLi$ for all $i$.
\end{cor}
\begin{proof}
Let $(E,V)\in G_L(r,d,k)$ be general. By \cite[Theorem 3.3(v)]{BGMMN1}, the bundle $E$ is stable and hence $(E,V)\in G_i(r,d,k)$ for all $i$. Under the assumption that  $\cL$ is general, the result follows from this fact and the theorem.
\end{proof}

\begin{theorem}\label{thm2}Suppose $g\ge2$, $d\ge r+1$ and $\cL$ possesses a section with distinct zeroes. Then $\GLr$ is non-empty and irreducible of dimension
$$\beta(r,d,r)-g+h^1(\cL).$$
\end{theorem}

Note that, by Bertini's Theorem, the hypothesis on $\cL$ is satisfied whenever $\cL$ is generated.

\begin{proof}[Proof of Theorem \ref{thm2}]
According to the proof of \cite[Theorem 5.6]{BGMN}, every $(E,V)\in \GLr$ can be expressed as a sequence
\begin{equation}\label{tor}
0\lra V\otimes\cO\lra E\lra T\lra0,
\end{equation}
where $T$ is a torsion sheaf whose underlying divisor is the divisor of a section of $\cL$. For fixed $T$, the dimension of the corresponding subspace of $\GLr$ is
$$\dim\mbox{Ext}^1(T,V\otimes\cO)-\dim\mbox{Aut}T-\dim\mbox{Aut}(V\otimes\cO)+1=rd-d-r^2+1=\beta(r,d,r)-d$$
(note that the centraliser of $(E,V)$ in $\mbox{Aut}T\times\mbox{Aut}(V\otimes\cO)$ is ${\mathbb C}^*$). The family of possible $T$ is irreducible of dimension $h^0(\cL)-1$, so  $\GLr$ is irreducible (if it is non-empty) of dimension 
$$\beta(r,d,r)-d+h^0(\cL)-1=\beta(r,d,r)-g+h^1(\cL).$$
It remains to prove that $\GLr$ is non-empty. For this, choose $T=\cO_D$, where $D$ is the divisor of a section of $\cL$ with $d$ distinct zeroes $P_1,\ldots,P_d$. An element of $\mbox{Ext}^1(T,V\otimes\cO)$ is then given by $(\xi_1,\ldots,\xi_d)$ with $\xi_i\in V$. Choose $\xi_i$ so that any subset of $r$ of these vectors is linearly independent. It is shown in the proof of \cite[Theorem 5.6]{BGMN} that the corresponding coherent system $(E,V)$ belongs to $G_L(r,d,r)$. This completes the proof.
\end{proof}

\begin{rem}\begin{em}
Note that $d\ge r+1$ is precisely the condition for the non-emptiness of $G_L(r,d,r)$ \cite[Theorem 5.6]{BGMN}.
\end{em}\end{rem}

\begin{cor}\label{cortor}
Under the hypotheses of Theorem \ref{thm2}, suppose further that $\cL$ is a general line bundle of degree $d$. Then, for general $(E,V)\in \GLr$, the bundle $E$ is stable and $(E,V)\in G_i(r,\cL,r)$ for all $i$.
\end{cor}
\begin{proof}
The proof is the same as for Corollary \ref{cortf}.
\end{proof}

\begin{rem}\begin{em}
It is possible that Corollaries \ref{cortf} and \ref{cortor} are valid without the generality assumption on $\cL$. To show this, one would need to prove a fixed determinant version of \cite[Theorem 3.3(v)]{BGMMN1}.\end{em}\end{rem}

\begin{theorem}\label{thm3}Suppose $g\ge2$ and $\cL$ is generated with $h^0(\cL)\ge r+1$. Then there exists a unique component $G_L$ of $\GLrr$ for which the general point $(E,V)$ is generated. Moreover
$$\dim G_L=(r+1)(h^0(\cL)-r-1)=\beta(r,d,r+1)-g+(r+1)h^1(\cL).$$
\end{theorem}
\begin{proof} The generated coherent systems $(E,V)$ with $\det E=\cL$ form an open subset $U$ of $\GLrr$; we need to show that $U$ is irreducible and of the required dimension.

In fact, given such $(E,V)$, we have an exact sequence 
\begin{equation}\label{gen}
0\lra\cL^*\lra V\otimes{\cO}\lra E\lra0.
\end{equation}
Dualising, we obtain
\begin{equation}\label{dual}
0\lra E^*\lra V^*\otimes{\cO}\lra \cL\lra0.
\end{equation}
Hence $\cL$ is generated by $V^*$ and, since $h^0(E^*)=0$ (see, for example, \cite[Lemma 2.9]{BGMMN1}), $V^*\subset H^0(\cL)$.  Conversely, given $V^*\subset H^0(\cL)$ of dimension $r+1$ and generating $\cL$, the sequences \eqref{dual} and \eqref{gen} define a coherent system $(E,V)\in U$ (for the stability of $(E,V)$, see \cite[Corollary 5.10]{BGMN}). Hence $U$ is isomorphic to some non-empty open subset of $\Gr(r+1,H^0(\cL))$. This proves irreducibility of $U$ and gives $\dim U=(r+1)(h^0(\cL)-r-1)$. The Riemann-Roch Theorem and the formula for $\beta(r,d,r+1)$ now give the result.
\end{proof}

We recall that a {\it Petri curve} is a curve $C$ for which the Petri map 
$$H^0(\cL)\otimes H^0(K_C\otimes \cL^*)\to H^0(K_C)$$
is injective for all line bundles $\cL$ on $C$.
A good description of $G_L(r,d,r+1)$ is known only when $C$ is Petri (see \cite[Theorem 5.11]{BGMN} and \cite[Theorem 3.1]{bbn}). In this case, it is known that, if $\beta(r,d,r+1)>0$, then $G_L(r,d,r+1)$ is irreducible (and smooth) of dimension $\beta(r,d,r+1)$. It is then sensible to ask whether $\GLrr$ is irreducible (and hence $\GLrr=G_L$). Note that $\beta(r,d,r+1)=\beta(1,d,r+1)$ and, if $C$ is Petri, then, by classical Brill-Noether theory, $B(1,d,r+1)$ is irreducible if $\beta(1,d,r+1)>0$ and is a finite set if $\beta(r,d,r+1)=0$.

\begin{cor}\label{cor31} Suppose that $C$ is a Petri curve, $g\ge2$, $0\le\beta(r,d,r+1)\le g$ and $\cL\in B(1,d,r+1)$ is general. (If $\beta(r,d,r+1)=0$, $\cL$ can be any element of the finite set $B(1,d,r+1)$.) Then $\GLrr=G_L$ and consists of a single element $(E_0, V_0)$.  Moreover $(E_0,V_0)\in \GLi$ for all $i$ and, except when $g=r=2$ and $d=4$, the bundle $E_0$ is stable.
\end{cor}
\begin{proof} We have $\beta(1,d,r+1)=\beta(r,d,r+1)$, so the condition on $\beta(r,d,r+1)$ implies that $h^0(\cL)=r+1$ and we must take $V^*=H^0(\cL)$ in \eqref{dual}; so $G_L$ is a single point. If $(E,V)\in \GLrr$ is not generated, let $E'$ be the subsheaf of $E$ generated by $V$. Then $H^0(E'^*)=0$ by \cite[Theorem 3.1(3)]{bbn} and $\det E'=\cL(-D)$ for some effective divisor of degree $t>0$. Applying \eqref{gen} to $(E',V)$ and dualising, we obtain $h^0(\cL(-D))\ge r+1$, contradicting the fact that $\cL$ is generated with $h^0(\cL)=r+1$. So $\GLrr=G_L$. The facts that $(E_0,V_0)\in \GLi$ for all $i$ and that $E_0$ is stable except when $g=r=2$ and $d=4$ are proved in \cite[Proposition 4.1]{bp} (see also \cite[Theorem 2]{but}).
\end{proof}

\begin{rem}\begin{em}
 Suppose that  $C$ is a Petri curve, $g\ge2$, $\beta(r,d,r+1)>g$ and $\cL\in B(1,d,r+1)=J^d(C)$ is general. Then $h^0(\cL)=d+1-g>r+1$, $h^1(\cL)=0$ and 
 $$\dim G_L=\beta(r,d,r+1)-g=(r+1)(d-r-g)>0.$$ 
 Moreover $\GLrr$ is irreducible and hence $\GLrr=G_L$.
 For general $V^*\subset H^0(\cL)$, one can ask whether  the bundle $E$ given by \eqref{dual} is stable (see \cite[Conjecture 9.5]{bbn}). The weaker conjecture that $E$ is semistable is a particular case of a conjecture of D. C. Butler  \cite[Conjecture 2]{but}  and has recently been proved \cite{bbn2} (see also \cite{afo}). Stability is true in many cases (see the list following Remark 9.7 in \cite{bbn} and the improvements to this list in \cite{bbn2}). Semistability is sufficient to show that $(E,V)\in G_i(r,\cL,r+1)$ for all $i$.
 \end{em}\end{rem}

\section{Examples}\label{ex} 

We have proved sharpness of the conjectured bound \eqref{conjeq} for $k\le r+1$ under a mild generality condition on $\cL$. The following example shows that some such condition is necessary.
\begin{ex}\begin{em}\cite[Theorem 1.1]{o1}
Suppose $r=2$ and $h^1(\cL)=0$. Let $\delta$ be the smallest degree of an effective divisor $\Delta$ such that $h^1(\cL(-\Delta))>0$. Then every irreducible component of $\GL$ has dimension at least
$$\beta(2,d,k)-g+\left(\begin{array}{c}k-\delta\\2\end{array}\right).$$
For example, if $\cL=K_C(x-y)$ with $x\ne y$, then $\delta=1$. Provided $k\ge3$, the bound of \eqref{conjeq} cannot be sharp.
\end{em}\end{ex}

Our second example, also noted by Osserman \cite[Example 5.2]{o2}, shows that, for $k=r+2$, something more is required.

\begin{ex}\begin{em} 
Let  $C$ be a general curve of genus $g=2a+1$ with $a\ge5$ and let $(r,d,k)=(2,2a+4,4)$. By \cite[Theorem 1.3]{fo}, there is a subvariety of $B(1,2a+4,5)$ of dimension $2$, whose general point is a line bundle $\cL$ with $h^0(\cL)=5$ such that $B(2,\cL,4)\ne\emptyset$ (in fact $B(2,\cL,4)$ is a point). However, $h^1(\cL)=1$, so Osserman's bound is $7-g$, which is negative. Osserman observes that one can allow for this by noting that $\beta(1,2a+4,5)=g-5$ and $(g-5)+(7-g)=2$; in other words, the bound has a relative validity. However, if one looks at individual line bundles $\cL$, an additional term is clearly required.
\end{em}\end{ex}

\end{document}